\documentclass[12pt,a4paper,final]{amsart}
\pdfoutput=1
\usepackage[utf8]{inputenc}
\usepackage[T1]{fontenc}
\usepackage[english]{babel}
\usepackage{algorithm}
\usepackage[noend]{algpseudocode}
\usepackage{amsfonts}
\usepackage{amsthm}
\usepackage{amsmath}
\usepackage{amscd}
\usepackage[mathscr]{eucal}
\usepackage{indentfirst}
\usepackage[final]{graphicx}
\usepackage{graphics}
\usepackage{pict2e}
\usepackage{epic}
\usepackage{mathtools}
\numberwithin{equation}{section}
\usepackage[margin=2.9cm]{geometry}
\usepackage{epstopdf}
\usepackage{tikz-cd}
\usepackage{csquotes}
\usepackage[
  backend=biber,
  style=alphabetic,
  maxalphanames=4,
  bibencoding=utf8,
  giveninits=true,
  maxbibnames=9,
  date=year,
  doi=true,isbn=false,url=false,eprint=true]{biblatex}
\usepackage{xparse}
\usepackage[hyperfootnotes=false,hypertexnames=false,final]{hyperref}
\usepackage[capitalise,noabbrev]{cleveref}
\AfterEndPreamble{
  \let\ref\cref
  \undef{\cref}%
  \undef{\autoref}%
}%

\theoremstyle{plain}
\newtheorem{Thm}{Theorem}[section]
\newtheorem{Lem}[Thm]{Lemma}

\newtheorem{Prop}[Thm]{Proposition}

 \theoremstyle{definition}
\newtheorem{Def}[Thm]{Definition}
\newtheorem{Conj}[Thm]{Conjecture}
\newtheorem{Rem}[Thm]{Remark}
\newtheorem{?}[Thm]{Problem}
\newtheorem{Ex}[Thm]{Example}
\newtheorem{Cl}[Thm]{Claim}
\newtheorem{OP}[Thm]{Open Problem}

\algrenewcommand\algorithmicrequire{\textbf{Input:}}
\algrenewcommand\algorithmicensure{\textbf{Output:}}

\DeclarePairedDelimiter\prn{\lparen}{\rparen}
\newcommand{\lr}[1]{\prn*{#1}}

\DeclareMathOperator{\im}{im}
\DeclareMathOperator{\kernel}{ker}
\DeclareMathOperator{\mom}{mom}
\newcommand{\Hom}{{\rm{Hom}}}

\newcommand{\E}{\mathbb{E}}
\newcommand{\N}{\mathbb{N}}
\newcommand{\Z}{\mathbb{Z}}
\newcommand{\C}{\mathbb{C}}
\newcommand{\R}{\mathbb{R}}
\newcommand{\dirac}[1]{\delta_{#1}}
\newcommand{\mommat}[1]{\mathrm{M}_{#1}}
\newcommand{\fac}[1]{#1!}
\newcommand{\e}{\mathrm{e}}
\newcommand{\dual}{^*}
\newcommand{\interval}[1]{\left[#1\right]}
\newcommand{\linterval}[1]{\left[#1\right[}
\usepackage{bbm}
\newcommand{\indicator}[1]{\mathbbm{1}_{#1}}
\renewcommand{\d}{\mathrm{d}}
\newcommand{\I}{\mathrm{i}}

\renewcommand{\P}{\mathbb{P}}
\NewDocumentCommand\mylist{O{:} >{\SplitList{,}}m}
{%
  \def\itemdelim{\def\itemdelim{#1}}
  \ProcessList{#2}{\mylistitem}
}
\newcommand\mylistitem[1]{\itemdelim #1}
\newcommand{\homog}[1]{%
  \let\dots\cdots{}
  \left[{\mylist[:]{#1}}\right]} 
\newcommand{\ratmap}{\dashrightarrow}
\newcommand{\units}{^{\times}}
\newcommand{\compose}{\mathbin{\circ}}
\newcommand{\closure}[1]{\overline{#1}}
\newcommand{\algclsr}[1]{\bar{#1}}
\newcommand{\restrict}[1]{|_{#1}}
\newcommand{\idealspan}[1]{\left\langle #1\right\rangle}
\newcommand{\saturation}[2]{#1\mathbin{:}{#2}^{\infty}}
\newcommand{\deriv}[1]{^{(#1)}}
\newcommand{\derivc}[1]{_{#1}}
\newcommand{\K}{\mathbbm{k}}
\newcommand{\variety}[1]{\mathcal #1}
\DeclareMathOperator{\rank}{rk}
\DeclareMathOperator{\initial}{in}
\newcommand{\contin}[1]{C^{#1}}
\newcommand{\convolution}{\mathbin{*}}

\title{Moment ideals of local Dirac mixtures}
\author{Alexandros Grosdos Koutsoumpelias\and Markus Wageringel}
\hypersetup{%
  pdftitle={Moment ideals of local Dirac mixtures},%
  pdfauthor={Alexandros Grosdos Koutsoumpelias, Markus Wageringel}}

\address{Institut f\"ur Mathematik,
Albrechtstr.~28\,a,
49076 Osnabr\"uck, Germany}
\email{\{alexandros.grosdoskoutsoumpelias, markus.wageringel\}@uos.de}



\begin{document}

\begin{abstract}
In this paper we study ideals arising from moments of local Dirac
measures and their mixtures.
We provide generators for the case of first order local Diracs
and explain how to obtain the moment ideal of the Pareto distribution from them.
We then use elimination theory and Prony's method
for parameter estimation of finite mixtures.
Our results are showcased with applications in signal processing and statistics.
We highlight the natural connections to algebraic statistics, combinatorics
and applications in analysis throughout the paper.
\end{abstract}

\maketitle

\section{Introduction} 

Moments of statistical and stochastic objects have recently gained attention 
from an algebraic and combinatorial point of view; 
see \cite{afs16:gaussians,rating,Kohn18}.
In this paper, we extend those methods to the study of 
moment ideals of mixture models coming from Dirac measures.

Finite mixture models appear in a wide range of applications in statistics
and possess a nice underlying geometric structure as in \cite{lindsay1983}.
They are of  use when a population consists of a finite number
of homogeneous subpopulations each having its own distribution
with density function $\phi_j(x)$.
Then the whole population follows a distribution with p.d.f.\ given by
\[\phi(x)=\sum_{j=1}^r \lambda_j \phi_j(x),\]
where $0\le\lambda_j \le 1$ and $\sum_{j=1}^r \lambda_j=1$.
A central problem associated with mixture models is identifying the parameters
involved in the distributions of the components as well as the mixing parameters $\lambda_j$ from a sample.
A common approach to this problem is computing the moments of the observed sample and finding the mixture model that best fits the observations. 

The moments $m_i$ of a distribution with probability density function $\phi$ are given by the integrals
\[m_i = \int x^i \phi(x) dx.\]
Moments of mixture distributions are therefore convex combinations of the moments of the components.
Despite the integral in the definition, 
it turns out that for many of the commonly used distributions
(Gaussian, Poisson, binomial, $\dots$),
the moments are polynomials in the parameters.
This allows for a number of algebraic techniques to be used,
such as studying determinants of moment matrices in \cite{lindsay1989},
or using polynomial algebra for the Gaussian distribution in work that started with Pearson \cite{pearson1894},
continued with \cite{monfrini2003} and \cite{lazard2004}
and was given a systematic algebraic and computational treatment in \cite{afs16:gaussians}.

In the case when the moments are polynomials
or, more generally, rational functions in the parameters,
one can study the (projective) variety containing all the points
$\homog{m_0, m_1, \dots , m_d} \in \P^d$.
Using Gr\"obner basis methods to compute the ideals
quickly becomes a computationally intractable problem
because these methods are very sensitive to the increase of the number of variables
as more mixture components are added.
On the geometric side, 
taking mixtures of a distribution corresponds to obtaining the secants of the  moment variety of this distribution, 
see for example \cite[Chapter~4]{drton2009}.
Aside from the statistical context, secant varieties play an important role in many other areas,
such as in tensor rank and tensor decomposition problems; see \cite{landsberg:tensors}.
A particular example of this is the symmetric tensor decomposition problem,
which can be formulated in terms of homogeneous polynomials and is classically known as Waring's problem:
Given a homogeneous polynomial $f\in \K[x_0,\dots,x_n]$ of degree $d$,
find a decomposition $f = L_{\xi_1}^d + \cdots + L_{\xi_r}^d$
into powers of linear forms $L_{\xi_j} \coloneqq \xi_{j0} x_0 + \cdots + \xi_{jn} x_n$, $j = 1,\dots,r$,
for the smallest possible natural number~$r$.
From a measure-theoretic point of view,
this involves a mixture of Dirac measures $\dirac{\xi_j}$ supported at points $\xi_j\in\P^n$.

Local mixtures are similar to mixture models and they have
a fascinating underlying geometric theory; see
\cite{marriott02:localmixtures,anaya-izquierdo2007}.
A finite local mixture of a distribution depending on a parameter $\xi$
with density function $\phi_{\xi}(x)$ involves 
adding some variation to the distribution through its derivatives
\[\phi_{\xi}(x) + \sum_{i=1}^l \alpha\derivc{i} \phi_{\xi}\deriv{i}(x) \]
for local mixing coefficients $\alpha\derivc{i}$.
We define $l$ to be the \emph{order} of a local mixture if $\alpha\derivc{l}$ is non-zero.
Adding to the statistics-to-geometry dictionary, 
the moments of local mixtures of order~$1$ correspond to taking the tangent variety of the moment variety.
Moments of higher-order local mixtures correspond to varieties known as osculating varieties \cite{bcgi2004:osculating}.

In this paper we consider the local mixtures of univariate Dirac measures.
We study the varieties associated to their moments 
and provide a generating set for the first order case in \ref{main2}
similar to the one in \cite{eis92:green}.
We use techniques from commutative algebra and combinatorics to prove this result.   
After reparametrizing the moments of the Pareto distribution, 
one observes that they are inverses of the first order local Dirac moments,
as shown in \ref{sec:pareto}.
Exploiting this fact and the generators we found, we obtain generators for the moment ideal of the Pareto as well.

In \ref{sec:recovery}, we study the problem of identifying the parameters of mixtures of first-order local Dirac distributions.
We use the equivalent cumulant coordinates that often 
simplify the varieties under consideration.
In the case of a mixture of two local Diracs
the moments are given by
\[m_i = \lambda(\xi_1^i+i\alpha_1\xi_1^{i-1})+(1-\lambda)(\xi_2^i+i\alpha_2\xi_2^{i-1}),\]
for a pair of parameters $(\xi_i, \alpha_i)$ from each mixture and a mixing parameter $\lambda$.
We are able to show that the moment (respectively cumulant) map 
\[ (\xi_1, \xi_2, \alpha_1, \alpha_2, \lambda) \longmapsto \homog{m_0, m_1, \dots , m_d} \]
is finite-to-one for $d=5$ and one-to-one for $d=6$. 
In addition to providing specific polynomials, 
we also formulate two algorithmic strategies 
for recovering the parameters of local mixtures.
The second algorithm is an extension of Prony's method,
which is a tool used in signal processing for recovering mixtures of Dirac distributions;
see \cite{plonka14:pronystructured} for a survey.

In \ref{sec:applications}, we illustrate the content of the paper numerically
by providing an application
to the reconstruction of piecewise-polynomial functions from Fourier samples.
For smooth functions, the truncated Fourier series provides a good approximation of a signal
due to rapid uniform convergence,
but in the presence of discontinuities,
uniform convergence is lost owing to the occurence of oscillations known as Gibbs' phenomenon
near the singularities.
Possible approaches to circumvent this include auto-adaptive spectral approximation \cite{mhaskarprestin2009}
and Gegenbauer approximation \cite{archibaldgelb2002}.
In contrast, we adopt a piecewise-polynomial model, which accounts for discontinuities.
Such models are frequently employed in schemes for detection
of discontinuities from sampled data \cite{lee1991,wright2010}.
Applying our results on parameter recovery of \ref{sec:recovery}
allows for the reconstruction of piecewise-polynomial functions
from the minimal number of Fourier samples.

The second application comes from local mixture models in statistics \cite{AnayaMarriott2007, anaya-izquierdo2007}.
These are used for samples whose variation has mostly -- but not entirely -- been explained by the model.
The local components added to the model account for the remaining variation.
We express local mixture models as convolutions of a basic model with a local Dirac.
We then proceed to present a numerical example where we estimate the mixing parameters 
using moments coming from a sample of local Gaussians.

\section{Preliminaries}

\subsection{Moments}

The $k$-th moment of the $l$-th order local mixture of a univariate Dirac measure $\dirac{\xi}$ is given by 
\[ m_{l,k} = \xi^k +\sum_{i=1}^{\min\{l,k\}} \alpha\derivc{i} \tfrac{\fac{k}}{\fac{(k-i)}} \xi^{k-i} \]
for some given parameters $\xi$ and $\alpha\derivc{1}, \dots , \alpha\derivc{l}$;
see \cite[Chapter~2]{schwartz1973}.
For example, for the first order mixture with $\alpha\coloneqq\alpha\derivc{1}$, we obtain 
$m_0=1$, 
$m_1=\xi+\alpha$, 
$m_2=(\xi+2 \alpha) \xi$,
$m_3=(\xi+3 \alpha) \xi^2$,
and so on.
Since models focus on some predetermined order $l$, we omit the $l$ in the subscript whenever the order is clear.

In statistics one has $m_0=1$, 
which leads to the restriction $\sum_{j=0}^l\lambda\derivc{i}=1$.
For a local mixture, this translates to the condition that the coefficient of the 
main component,
in this case the coefficient of $\xi^k$, is one;
see also \cite{anaya-izquierdo2007}.
Note that since the components of a local mixture are not necessarily probability distributions
one does not require $\alpha\derivc{i} \geq 0$,
but other semialgebraic restrictions are necessary for statistics,
see for instance \ref{sec:applications2} and \cite{marriott02:localmixtures}.


\begin{Def}
  Let $f_0,\dots,f_d\in\K(x_1,\dots,x_s)$ be rational functions.
  The Zariski-closure of the set parametrized by
  \[
    \{\homog{m_0,\dots,m_d} \in \P^d \mid m_i = f_i(x)\text{ for }x\in\K^s,\ 0\le i\le d \}
  \]
  is a projective variety of dimension at most $s$.
  If the $f_i$ are expressions for the moments of a family of distributions with $s$ parameters,
  every point in the parametrically given set is a \emph{moment vector} of an element of the family
  and the variety is called \emph{moment variety} with respect to $d$ of the family of distributions,
  which we commonly denote by $\variety{X_d}$.
\end{Def}
Throughout this paper, we will denote the homogeneous ideal of $\variety{X_d}$ by $I_d$.
Since in statistics the zeroth moment $m_0$ is equal to 1,
we will often work with the corresponding dehomogenized version of the ideal which we denote by $\tilde{I_d}$.

Even though the distributions we consider are usually defined over the real numbers~$\R$,
since the moments are polynomials or rational functions in the parameters,
it is often convenient to work with the complexification of the moment varieties,
or more generally, to work over any algebraically closed field.
Therefore, unless specified otherwise,
we will assume $\K$ to be an algebraically closed field of characteristic~$0$.

\subsection{Difference functions}

For describing the generators of the moment ideal,
the following notion is useful.
We define $\Delta_r$ as the Vandermonde determinant \[
  \Delta_r = \det \left( X_j^k \right)_{0\le k,j\le r}
  = \prod_{0\le i<j \le r} (X_j - X_i),
\] considered as an element of the polynomial ring $\K[X_0,\dots,X_r]$.
Thus, for $r=1$, the powers of $\Delta_1$ are just the higher order differences
\begin{align*}
  \Delta_1 &= X_1 - X_0,\\
  \Delta_1^2 &= X_1^2 - 2X_0 X_1 + X_0^2,\\
  \Delta_1^3 &= X_1^3 - 3X_0 X_1^2 + 3X_0^2 X_1 - X_0^3.
\end{align*}
We define $\E$ to be the $\K$-linear map between polynomial rings
\begin{equation}\label{mapEE}
  \begin{aligned}
  \E\colon \K[X_0,\dots,X_r] &\longrightarrow \K[M_0,M_1,\dots],\\
  X_0^{a_0}\cdots X_r^{a_r} &\longmapsto M_{a_0}\cdots M_{a_r},
  \end{aligned}
\end{equation}
in which we interpret
$X_0,\dots,X_r$ as abstract random variables
which are viewed as independent replicates of a distribution,
that is,
\[
  \E(X_j^a X_{j'}^{a'}) = \E(X_j^a) \E(X_{j'}^{a'}) = M_a M_{a'}.
\]
The map $\E$ can then be understood as the expectation of random variables,
mapping any random variable
that is a polynomial expression in the variables $X_0,\dots,X_r$
to the abstract moments,
which are polynomial expressions in $M_a$, $a\in\N$.
More formally, this interpretation is captured by the concept of Umbral Calculus;
see, e.\,g.,\ \cite{rota00:umbral} for a brief overview.

With this notation, \[
  \E(X_0^{a_0} \cdots X_r^{a_r}\Delta_r^n)
\] is a homogeneous polynomial of degree $r+1$ in the moments $M_a$, $a\in\N$.
For example, for $r=1$, we get \[
  \E(X_0^{a_0} X_1^{a_1}\Delta_1^n) = \sum_{k=0}^{n} (-1)^k \binom{n}{k} M_{a_0+k}M_{a_1+n-k}.
\]
Note that, in the notation of \cite{eis92:green}, we have
$\E(X_0^{a_0} X_1^{a_1}\Delta_1)=\Delta_{a_0,a_1}$ as well as
$\E(X_0^{a_0} X_1^{a_1}\Delta_1^3)=\Gamma_{a_0,a_1}$.

When working with finitely many variables $M_0,\dots,M_d$,
we will assume that the map $\E$ is restricted to a suitable subspace.


\section{Ideals of local mixtures}

\subsection{Generators of the first order moment ideal}
\label{sec:firstorderdiracs}

In this section, we focus on the case of local mixtures of order~$l=1$.
Our main goal is to find a generating set for the ideal
$I_d \coloneqq \variety{I} (\variety{X}_d)$,
the homogeneous (with respect to the standard grading) ideal of the moment variety. 
Note that this ideal is also homogeneous with respect to
the grading induced by the weight vector $(0,1, \dots , d)$.
The main result we prove here is the following:

\begin{Thm}
\label{main2}
For $d \geq 6$, let $J_d$ be the ideal generated by the $\binom{d-2}{2}$ relations
\[
f_{i,j} \coloneqq (j-i+3)M_iM_{j}-2(j-i+2)M_{i-1}M_{j+1}+(j-i+1)M_{i-2}M_{j+2}, 
\]
for $2 \leq i \leq j \leq d-2$.
Then $J_d$ is equal to $I_d$, the homogeneous ideal of the moment variety.
\end{Thm}

We remark here that an alternative equivalent set of generators was given in \cite[Section~3]{eis92:green}:
\begin{Thm}
\label{main1}
For $d\ge 6$,
\[
  I_d = \langle \E(X_0^{a_0} X_1^{a_1}\Delta_1^3) \mid 0 \leq a_0 < a_1 \leq d-3 \rangle
\]
is the ideal
generated by $\binom{d-2}{2}$ relations coming from the third powers of Vandermonde determinants.
\end{Thm}

More explicitly, \ref{main1} means that the moment ideal
is generated by the $\binom{d-2}{2}$ quadratic relations
\begin{equation}
\label{eiseq}
M_iM_{j+3}-3M_{i+1}M_{j+2}+3M_{i+2}M_{j+1}-M_{i+3}M_j \text{ for } 0 \leq i < j \leq d-3.
\end{equation}

The proof in Eisenbud's paper employs multilinear algebra and representation theory 
to find the generators of the ideal of the variety.
Our proof relies heavily on combinatorics
in order to compute the Hilbert functions of the ideals involved,
as explained below.

In order to prove \ref{main2}, we employ the following strategy.
First, we work with the dehomogenized version $\tilde{J_d}$ and $\tilde{I_d}$ of the ideals,
where we set $M_0=1$.
The ideal $\tilde{J_d}$ can be seen to be contained in $\tilde{I_d}$,
which is shown in \ref{eq:momenteval} below.
Then, we use the grading of the polynomial ring $\K[M_1, M_2, \dots , M_d]$ 
given by the vector $(1, 2, \dots, d)$
as well as combinatorics to show that the ideals in question 
have the same Hilbert series.

\begin{Lem}
\label{SisininJ} 
Let $\prec$ be  the monomial order on $\K[M_1,M_2, \dots , M_d]$ 
that compares monomials 
with the reverse lexicographical order 
with $M_1 \prec M_2 \prec \dots \prec M_d$.
Then the monomial ideal $S_d$ generated by  
\begin{equation*}
\langle M_iM_j \mid 2 \leq i \leq j \leq d-2 \rangle + \langle M_1M_iM_{d-1} \mid 2 \leq i \leq d-2 \rangle + \langle M_1^2M_{d-1}^2 \rangle
\end{equation*}
is contained in $\initial_{\prec}J_d$.
\end{Lem}

\begin{proof}
The degree 2 monomials of $S_d$ are precisely 
the leading terms of the generators of $J_d$ with respect to $\prec$.
For the degree 3 monomials in the generating set of $S_d$, one obtains $M_1M_iM_{d-1}$ for $i=2,3, \dots , d-3$ as the leading term of the S-pair $S(f_{2,i},f_{2,d-2})$ reduced by the elements in the generating set of $J_d$.
Similarly, the monomial $M_1M_{d-2}M_{d-1}$ arises as the leading term of the S-pair $S(f_{2,d-2},f_{d-2,d-2})$ after reducing.
Finally, $M_1^2M_{d-1}^2$ is the leading term 
of a polynomial in the ideal $\langle f_{2,d-2}, f_{2,d-3}, f_{d-2,d-2} \rangle $.
One obtains this polynomial by taking the S-pair of $f_{2,d-2}$ and $f_{d-2,d-2}$, 
reducing by $f_{2,d-3}$, 
and then taking the S-pair of the resulting polynomial with $f_{2,d-2}$.
\end{proof}

Our goal now is to show that the ideal $S_d$ has the same Hilbert series as $\tilde{I_d}$ from which we conclude that it is indeed its initial ideal.

\begin{Lem}
\label{Mprimdec}
The ideal $S_d$ defined in \ref{SisininJ} has a primary decomposition 
\begin{equation}
\langle M_iM_j \mid 1 \leq i \leq j \leq d-2 \rangle \cap \langle M_iM_j \mid 2 \leq i \leq j \leq d-1 \rangle.
\end{equation}
\end{Lem}

\begin{proof}
Set $C=\langle M_iM_j \mid 2 \leq i \leq j \leq d-2 \rangle$.
Then we want to show that
\begin{equation*}
S_d = (\langle M_1M_i \mid 1 \leq i \leq d-2 \rangle)+C) \cap (C+ \langle M_iM_{d-1}\mid 2 \leq i \leq d-1 \rangle).
\end{equation*}

Inclusion of $S_d$ in the other ideal follows by checking that
each generator of $S_d$ is divided by some monomial in each part of the decomposition.

For the other inclusion, take a monomial $m$ in the intersection.
If the monomial is divided by some monomial in $C$, it is clearly in $S_d$.
If not, there are two possibilities.
The first one is that $m$ is divided by both $M_1^2$ and $M_{d-1}^2$, in which case it is divided by the single generator $M_1^2M_{d-1}^2$ of $S_d$.
The other one is that $m$ is divided by two monomials $M_1M_i$ and $M_jM_{d-1}$, with 
$j\geq 2$ and $i \leq d-2$.
In this case, $M_1M_iM_{d-1}$ and $M_1M_jM_{d-1}$ both divide $m$ and at least one of them must be in $S_d$.
\end{proof}

Consider the polynomial moment map
\begin{equation}
\label{mommap}
\begin{aligned}
\mom_d\colon \K[M_1, M_2, \dots , M_d] &\longrightarrow \K[A,X]    \\
M_i &\longmapsto (X+iA)X^{i-1}.
\end{aligned}
\end{equation}
From now on, we consider the grading given by $\deg M_i=i$ 
on $\K[M_1, M_2, \dots , M_d]$
and the standard grading $\deg A = \deg X = 1$ on $\K[A,X]$.
The moment map becomes graded this way.
The ideal $\tilde{I_d}$ is the kernel of this map.
Using the lemmata above, we are able to compare the Hilbert series of the ideals.
 
\begin{proof}[Proof of \ref{main2}]
We show first the equality of the dehomogenized ideals  $\tilde{I}_d$ and  $\tilde{J}_d$.
The inclusion $\tilde{J_d} \subseteq \tilde{I_d}$ follows from the fact that,
for each of the generators of $\tilde{J_d}$,
substituting $M_i$ with its image given by the moment map  evaluates to zero.
Indeed, we have
\begin{equation}\label{eq:momenteval}
\begin{aligned}
& \mom_d((j-i+3)M_iM_{j}-2(j-i+2)M_{i-1}M_{j+1}+(j-i+1)M_{i-2}M_{j+2})  \\
={} &((j-i+3)-2(j-i+2)+(j-i+1))((X^{i+j}+(i+j)AX^{i+j-1}+ijA^2X^{i+j-2}) \\
={} & 0.
\end{aligned}
\end{equation}

We now show equality.
Let $\mom_d$ be the moment map~\labelcref{mommap}
and consider the short exact sequences
\begin{equation*}
0 
\longrightarrow \tilde{I_d} 
\longrightarrow \K[M_1, M_2, \dots , M_d] 
\longrightarrow \im \mom_d 
\longrightarrow 0
\end{equation*}
and
\begin{equation*}
0 
\longrightarrow S_d 
\longrightarrow \K[M_1, M_2, \dots , M_d] 
\longrightarrow  \K[M_1, M_2, \dots , M_d]/S_d
\longrightarrow 0.
\end{equation*}

\begin{Cl} \label{cl1}
Let $d \geq 3$.
For $ n \geq 1$, the vector space in degree $n$ in the image of the moment map has dimension $n$.
Thus, the Hilbert series of the image of the moment map is 
\begin{equation*}
\text{HS} (\im \mom_d) = \frac{1-t+t^2}{(1-t)^2}. 
\end{equation*}
\end{Cl}

\begin{Cl} \label{cl2}
Let $d \geq 6$.
The Hilbert series of the graded algebra $\K[M_1, M_2, \dots , M_d]/S_d$ is
\begin{equation*}
\text{HS} (\K[M_1, M_2, \dots , M_d]/S_d) = \frac{1-t+t^2}{(1-t)^2}. 
\end{equation*}
\end{Cl}
Since $S_d \subseteq \initial_{\prec}\tilde{J_d} \subseteq \initial_{\prec}\tilde{I_d}$, for all $n$ we have the inequalities
\begin{equation} \label{eq462}
\text{HF}(S_d)(n) \leq \text{HF}(\initial_{\prec}\tilde{J_d})(n) \leq \text{HF}(\initial_{\prec}\tilde{I_d})(n) = \text{HF}(\tilde{I_d})(n). 
\end{equation}
Since the two Hilbert functions of $\im \mom_d $ and $\K[M_1, M_2, \dots , M_d]/S_d$ coincide, 
it follows from the two exact sequences above that $\text{HF}(S_d)(n)$ and $\text{HF}(\tilde{I_d})(n)$ are also equal.
Thus, all inequalities in~\labelcref{eq462} are in fact equalities, 
implying that $\tilde{J_d}=\tilde{I_d}$.

By \cite[Section~8.4, Theorem~4]{Cox:2015},
it follows that the Gr\"obner basis of $I_d$ is the homogenized version of the Gr\"obner basis for $\tilde{I_d}$.
This can be obtained in both cases by using the Buchberger algorithm on the corresponding generating set given by \ref{main2} 
to obtain polynomials whose initial terms are the ones given by \ref{SisininJ}. 
The homogeneous version follows, i.\,e., $I_d$ is equal to $J_d$.
\end{proof}

We finally show the two claims we used in the proof of the \namecref{main2}.

\begin{proof}[Proof of \ref{cl1}]
We use induction to show that, for each degree $n$ (recall that we are using the standard grading here),
the corresponding vector space has $n$ basis elements
$X^n+nAX^{n-1}$ and $A^iX^{n-i}$ for $i = 2,3, \dots , n$ and $n \geq 2$.

For $n=1$, the corresponding vector space is generated by $X+A$.

For $n=2$, the only generators are $(X+A)^2$ and the image of $M_2$, that is $X^2+2AX$.
It follows that the vector space in degree $2$ is generated by $X^2+2AX$ and $A^2$.

For $n=3$, there are three generators $X^3+3AX^2$, $(X+A)A^2$ and $(X+A)^3$.
One checks that $(X+A)^3-(X^3+3AX^2)=3A^2X+A^3$ and $(X+A)A^2=XA^2+A^3$, which implies the vector space has a basis consisting of the elements $X^3+3AX^2$, $A^2X$ and $A^3$.

Now assume that the inductive hypothesis is true in degrees $n$ and $n+1$.
Then the basis elements $A^iX^{n+2-i}$ for $i = 4,5, \dots , n$ in degree $n+2$ arise by multiplying the corresponding terms in degree $n$ with $A^2$.
Further, multiplying the terms $A^3X^{n-2}$, $A^2X^{n-1}$ and $X^{n+1}+(n+1)AX^n$ of degree $n+1$
with $X+A$, we obtain $A^3X^{n-1}+A^4X^{n-2}$, $A^2X^{n}+A^3X^{n-1}$ and $X^{n+2}+(n+2)AX^{n+1}+(n+1)A^2X^n$,
which give rise to the remaining three required terms.

It remains to show that there exists no combination
other than $X^{n+2}+(n+2)AX^{n+1}$ 
that involves the terms $X^{n+2}, AX^{n+1}$.
Indeed, any such combination can only arise in the following way: 
we choose $k$ and $l$ between $0$ and $n+2$ such that $k+l=n+2$ and we form the product 
$(X^k+kAX^{k-1})(X+A)^l$.
But this is equal to $X^{n+2}+(n+2)AX^{n+1}+ \text{sum of lower terms}$.
Thus the given set of polynomials is indeed a vector space basis.
\end{proof}

\begin{proof}[Proof of \ref{cl2}]
We use the primary decomposition of the monomial ideal $S_d$ given by \ref{Mprimdec} to compute the Hilbert series. 
By inclusion-exclusion on the primary decomposition, we obtain that $\text{HS} (\K[M_1, M_2, \dots , M_d]/S_d)$ is equal to
\begin{align*}
  &\text{HS} (\K[M_1, M_2, \dots , M_d]/\langle M_iM_j \mid 1 \leq i \leq j \leq d-2 \rangle) \\
  {}+{}&\text{HS} (\K[M_1, M_2, \dots , M_d]/\langle M_iM_j \mid 2 \leq i \leq j \leq d-1 \rangle) \\
  {}-{}&\text{HS} (\K[M_1, M_2, \dots , M_d]/\langle M_iM_j \mid 1 \leq i \leq j \leq d-1 \rangle).
\end{align*}
We compute the three Hilbert series separately.

For the first one, note that the vector space in degree $n$ is spanned by monomials in degree $n$
that use the indeterminates $M_{d-1}$, $M_d$ and maybe one of the $M_i$ for $1 \leq i \leq d-2$ at most once 
(indeed, any pairwise product of them vanishes in the quotient.)
Thus, the cardinality of the vector space is equal to the number of ways of partitioning the number $n$ as a sum
using only the numbers $d-1$, $d$ and at most one of the numbers $1, 2, \dots, d-2$ at most once.
Hence, the first Hilbert series is 
\begin{equation*}
\frac{1}{1-t^{d-1}} \frac{1}{1-t^d}(1+t+\dots + t^{d-2}).
\end{equation*}
A similar argument for the second part yields the series
\begin{equation*}
\frac{1}{1-t} \frac{1}{1-t^d}(1+t^2+t^3+\dots + t^{d-1}),
\end{equation*}
while for the third one we obtain
\begin{equation*}
\frac{1}{1-t^d}(1+t+\dots + t^{d}).
\end{equation*}
Summing up gives :
\begin{align*}
&\frac{1}{1-t^d}\lr{\frac{1}{1-t^{d-1}}(1+\dots + t^{d-2})+\frac{1}{1-t}(1+t^2+t^3+\dots + t^{d-1}) -(1+\dots + t^{d})} \\
={}&\frac{1}{1-t^d}\frac{1}{(1-t)^2}\lr{(1-t) +(1-t)+t^2-t^d-(1-t)(1-t^{d+1})} \\
={}&\frac{1-t+t^2}{(1-t)^2}.
\end{align*}
This finishes the proof of the claimed result.
\end{proof}

\begin{Rem}
We have computed above the Hilbert series of our ideals for a special grading.
Eisenbud in \cite{eis92:green} proved the equivalent result for the standard grading.
For $ d \geq 6$, the Hilbert series of the ideal $I_d$ is
\[\frac{1+(d-2)t+(d-2)t^2+t^3}{(1-t)^3}.\]
\end{Rem}

\begin{Rem}
We apply the methods of \cite{ccmrz2016:cremona}
on Cremona linearizations to simplify the description of the moment ideal.
For this, let $y_0\coloneqq M_0 = 1$ as before
and define a polynomial transformation as \[
  y_i\coloneqq \begin{cases}
    M_i &\text{if $i\le 3$},\\
    M_i - z_i &\text{otherwise},
  \end{cases}
\]
where \[
  z_i\coloneqq \begin{cases}
    M_i &\text{if $i\le 3$},\\
    \frac{1}{2} k(k+1) z_{k-1} z_{k+2} - \frac{1}{2} (k-1)(k+2) z_k z_{k+1} &\text{if $i>3$ and $i = 2k+1$},\\
    k^2 z_{k-1} z_{k+1} - (k-1)(k+1) z_k^2 &\text{if $i>3$ and $i = 2k$}.
  \end{cases}
\]
This map is triangular and thus invertible.
Further, one checks that
in these variables, if $d\ge 3$, the moment variety is defined by the quartic polynomial equation
\begin{equation}\label{eq:discriminant}
  3 y_1^2 y_2^2 - 4 y_1^3 y_3 - 4 y_2^3 + 6 y_1 y_2 y_3 - y_3^2 = 0
\end{equation}
and the equations $0=y_4=y_5=\cdots =y_d$.
This means that the variety is mapped isomorphically
to a surface in a three-dimensional linear subspace.

Note that the polynomial in~\ref{eq:discriminant} is, up to a constant factor,
the discriminant of the polynomial \[
  1 + 3y_1 X + 3y_2 X^2 + y_3 X^3.
\]
Since the discriminant of a polynomial vanishes whenever it has at least one double root,
the moment variety $\variety{X}_3$ is thus parametrized by cubics with a double root.
Recall that this is the tangent variety of the Veronese curve.
It is depicted in \ref{fig:discriminant}.

\begin{figure}[ht]
    \centering
    \includegraphics[width=48mm]{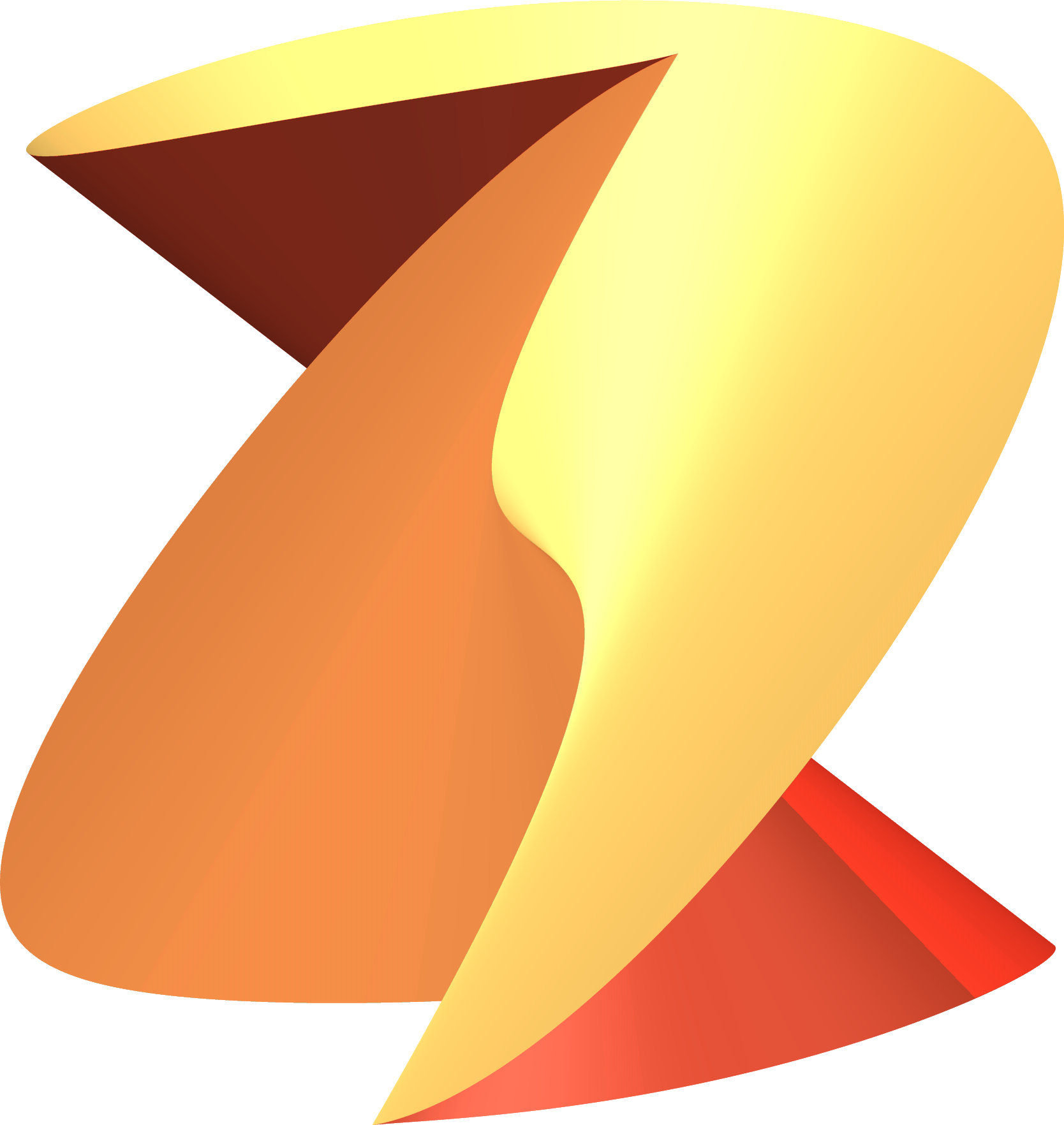}
    \caption{The surface defined by \ref{eq:discriminant}}
    \label{fig:discriminant}
\end{figure}

More generally, up to a projective linear transformation,
the tangent variety of the Veronese curve has the parametrization
\[
  \{ \homog{u^{d-1} v} \in \P^d : u,v\text{ linear forms} \}.
\]
We refer to \cite[Corollary, p.~305]{oedraicu14:tangential}
for the general description of the generators
for the tangential variety of the Veronese variety that also covers the multivariate case,
as well as a generalization to tangential varieties of Segre-Veronese varieties.
\end{Rem}

\subsection{Conjectures for higher orders}

Note that by replacing $i$ with $i+2$, 
the generators of the polynomial in \ref{main2} can be written in the more symmetric form
$(j-i+1)M_{i+2}M_{j}-2(j-i)M_{i+1}M_{j+1}+(j-i-1)M_{i}M_{j+2}$.
Computer experiments with Macaulay2 \cite{M2} seem to suggest the following extrapolation for a $2$-local mixture:

\begin{Conj}
Let $I_{2,d}$ be the ideal of the moments of the second-order local mixture. Then for $d \geq 12$ this ideal is generated by 
\[c_0M_{i+3}M_j+c_1M_{i+2}M_{j+1}+c_2M_{i+1}M_{j+2}+c_3M_{i}M_{j+3}\]
for $i \geq 0$, $j \geq 0$ and $i \geq j-3$, where 
\begin{align*}
c_0&= (j-i+1)(j-i+2) & c_1&=-3(j-i-1)(j-i+2)     \\
c_2&= 3(j-i+1)(j-i-2) & c_3&=-(j-i-1)(j-i-2).        
\end{align*}
The equivalent generators in the notation of \ref{main1} are
\begin{multline*}
\E(X_0^{i} X_1^{j}\Delta_1^5) = \\
M_iM_{j+5}-5M_{i+1}M_{j+4}+10M_{i+2}M_{j+3}-10M_{i+3}M_{j+2}+5M_{i+4}M_{j+1}-M_{i+5}M_j.
\end{multline*}
\end{Conj}
Although the second set of generators seems more natural than the first one above,
it appears that the methods in \cite{eis92:green} do not immediately extend to the higher order case.
On the other hand, one could retrace the steps of the proof of \ref{main2} to prove this Conjecture.
The main difficulty here 
is finding the dimension of the vector spaces in the image of the moment map
$M_i \mapsto X^i+AiX^{i-1}+B^2 i(i-1) X^{i-2}$.

One can generalize this set of generators for higher orders of mixtures:
\begin{Conj}
Let $I_{l,d}$ be the ideal of the moments of the $l$-th order local mixture.
Then for $d$ sufficiently large, this ideal is generated by the polynomials
\[
  \E(X_0^{a_0} X_1^{a_1}\Delta_1^{2l+1}) =
  \sum_{k=0}^{2l+1} (-1)^k \binom{2l+1}{k} M_{a_0+k}M_{a_1+2l+1-k}.
\]
\end{Conj}

\section{Pareto distribution}
\label{sec:pareto}

The Pareto distribution was introduced by Vilfredo Pareto 
as a model for income distribution, see \cite{arnold1983}.
It is a heavy-tailed continuous probability distribution 
that finds a wide range of applications, especially in econometrics.
In the univariate case, its probability density function is
\[
  \varphi(x) \coloneqq \frac{\alpha\xi^{\alpha}}{x^{\alpha+1}} \indicator{\{x\ge\xi\}},
\]
where $\alpha,\xi \in\R_{>0}$.
The moments of this distribution are given by
\[
  m_i = \begin{cases}
    \frac{\alpha}{\alpha-i} \xi^i, & i < \alpha,\\
    \infty, & i \ge\alpha;
  \end{cases}
\]
see \cite{arnold1983}.
We show below how to reparametrize them so that the moments of the Pareto are the inverses of the first order local mixture 
of Diracs
and exploit this fact to obtain generators of the Pareto moment ideal.

\subsection{Ideal generators}

Algebraically, we are interested in the cases for which the moments are finite,
which are described by the image of the map
\begin{align*}
  \R_{>d} \times \R_{>0} &\longrightarrow \P^d,\\
  (\alpha,\xi) &\longmapsto \homog{m_0,\dots,m_d} = \homog{\frac{\alpha}{\alpha-i} \xi^i}_{0\le i\le d},
\end{align*}
for a given $d\in\N$, where $\P^d$ denotes the projective space over $\C$.
We define the moment variety of the Pareto distribution as the Zariski-closure over $\C$
of the image of the above map.
Since $\R_{>d}$ is Zariski-dense in $\C$, we may extend the domain of the parametrization to
$(\C\setminus\{0,\dots,d\}) \times \C\units$
without changing the Zariski-closure of the image.
Let $\rho$ be the corresponding map $\rho\colon (\C\setminus\{0,\dots,d\}) \times \C\units \to \P^d$.

\begin{Prop}
\label{birational}
  Let $\variety{Y} \coloneqq \closure{\im(\rho)} \subseteq\P^d$ be the Pareto moment variety
  and $\variety{X}\subseteq\P^d$ the moment variety of $1$-local mixtures,
  that is, the tangent variety of the Veronese curve.
  Then $\variety{X}$ and $\variety{Y}$
  are birationally equivalent via the rational map
  \[
    \psi\colon \P^d \ratmap \P^d, \quad \homog{m_0,\dots,m_d} \longmapsto \homog{m_0^{-1},\dots,m_d^{-1}}.
  \]
  \begin{proof}
    We change the parametrization of the Pareto moment variety via the bijective map
    \begin{align*}
      \eta\colon \{ (\alpha,\xi) \in \C\units \times \C\units \mid
        -\xi \alpha^{-1} \neq 1,\dots,d
      \} &\longrightarrow
      (\C\setminus\{0,\dots,d\}) \times \C\units, \\
      (\alpha,\xi) &\longmapsto \left(-\xi \alpha^{-1}, \xi^{-1}\right),
    \end{align*}
    which leaves $\variety{Y}$ as the closure of the image of $\rho\compose\eta$ unchanged.
    With this parametrization, the moments are of the form
    \[
      \homog{m_0,\dots,m_d} = \rho(\eta(\alpha,\xi))
      = \homog{\frac{-\xi\alpha^{-1}}{-\xi\alpha^{-1} - i}\ \xi^{-i}}_{0\le i\le d}
      = \homog{\frac{1}{\xi^i + i \alpha \xi^{i-1}}}_{0\le i\le d},
    \]
    so $\psi$ maps points from the image of $\rho\compose\eta$
    to moment vectors of $1$-local mixtures, that is, to points on $\variety{X}$.
    Then $\psi\restrict{\variety{Y}}\colon \variety{Y}\ratmap \variety{X}$ is a rational map
    that is an isomorphism on the dense subset
    $\im(\rho\compose\eta)$ of $\variety{Y}$, as $\im(\rho\compose\eta) \subseteq \{m_i \neq 0\}$.
    Being the tangent variety of an irreducible variety,
    $\variety{X}$ is irreducible by \cite[Section~8.1]{landsberg:tensors}.
    Thus,
    the image of $\psi\restrict{\variety{Y}}$ is dense in $\variety{X}$ which proves the claim.
  \end{proof}
\end{Prop}

\begin{Thm}\label{thm:paretogenerators}
  For $d \geq 6$, let $\tilde{I}_d^{\textnormal{inv}} \subseteq R\coloneqq \C[M_1,\dots,M_d]$
  be the ideal generated by the $\binom{d-2}{2}$ polynomials
  \begin{equation*}
    (j-i+3)M_{i-2}M_{i-1}M_{j+1}M_{j+2}-2(j-i+2)M_{i-2}M_{i}M_{j}M_{j+2}+(j-i+1)M_{i-1}M_{i}M_{j}M_{j+1}
  \end{equation*}
  for $2 \leq i \leq j \leq d-2$, where we assume $M_0=1$.
  Then the affine Pareto moment ideal is equal to the saturation
  \[
    \saturation{\tilde{I}_d^{\textnormal{inv}}}{(M_1\cdots M_d)}.
  \]
\end{Thm}
\begin{proof}
  Let $\tilde{I}_d \subseteq \C[M_1,\dots,M_d]$ be the dehomogenization of
  the moment ideal of local mixtures of Diracs
  which was studied in \ref{sec:firstorderdiracs}.
  In order to restrict to the algebraic torus where the rational map $\psi$
  given in \ref{birational} is defined,
  consider $J\coloneqq R[y] \tilde{I}_d + \idealspan{M_1\cdots M_d y - 1} \subseteq R[y]$.
  The restriction of the map $\psi$ to the torus agrees with
  the torus automorphism induced by the homomorphism
  \begin{align*}
    \bar{\psi}\colon R[y] &\longrightarrow R[y],\\
    y &\longmapsto M_1\cdots M_d,\\
    M_i &\longmapsto M_1\cdots M_{i-1} M_{i+1}\cdots M_d y.
  \end{align*}
  Note that we can choose an ideal $I'\subseteq R$ such that
  $\bar{\psi}(J) = R[y] I' + \idealspan{M_1\cdots M_d y -1}$
  by observing that, for any $f\in \tilde{I}_d$,
  we can choose a suitable $k\in\N$ such that
  $(M_1\cdots M_d)^k \bar{\psi}(f) \equiv f'\pmod{\idealspan{M_1\cdots M_d y -1}}$
  for some $f'\in R$.
  In particular, this construction establishes a bijection between the generating set of $\tilde{I}_d$
  given in \ref{main2}
  and the generating set of $\tilde{I}_d^{\textnormal{inv}}$.
  Therefore, we choose $I'\coloneqq \tilde{I}_d^{\textnormal{inv}}$.
  In order to describe the affine closure of the image of $\psi$,
  we intersect $\bar{\psi}(J)$ with $R$,
  which is equal to
  \[
    \bar{\psi}(J)\cap R =
    \saturation{\tilde{I}_d^{\textnormal{inv}}}{(M_1\cdots M_d)}
  \]
  by \cite[Theorem~4.4.14]{Cox:2015} from which we conclude.
\end{proof}

The statement of the \namecref{thm:paretogenerators} also holds
for the ideal we get
if, in the construction, we replace the generators of \ref{main2} by those in \labelcref{eiseq}.
Taking the saturation in the construction is necessary
in order to prevent the variety from containing additional irreducible components
that are supported on the boundary of the algebraic torus,
which cannot be part of the Pareto moment variety.
In computations we carried out, 
the ideal $\tilde{I}_d^{\textnormal{inv}}$ 
has a smaller primary decomposition 
and performs faster with computer algebra systems
when compared with its counterpart implied by \labelcref{eiseq}.

\section{Recovery of Parameters}
\label{sec:recovery}

In this section we use the method of moments \cite{Bowman06} 
to estimate the parameters of mixtures of Dirac local mixture distributions.
In statistics, one often starts with a measurement from a population or a sample following a particular distribution.
From this, one can compute the empirical moments (or equivalently the cumulants)
and try to deduce the original parameters of the underlying distribution.
By contrast, in signal processing, one usually obtains the moments of a measure directly using Fourier methods.

\subsection{Cumulants}

So far in this paper, we focused on moment coordinates.
Cumulants are an alternative set of associated quantities 
that are well-known to statisticians
and have recently become an object of study for algebraists.
The cumulants $k_i$ of a distribution can be given as an invertible polynomial transformation of the moments and they carry the same information. 
However, for many interesting distributions studied in the literature, 
they have a simpler form than the moments 
and doing computations with them empirically turns out to be faster,
which becomes very useful when Gr\"obner basis computations are involved.
In the univariate case, let 
\[M(t)=\sum_{i \geq 0}^{\infty}m_i\frac{t^i}{\fac{i}} \quad \text{ and } \quad  K(t)=\sum_{i \geq 0}^{\infty}k_i\frac{t^i}{\fac{i}}\] 
be the moment and cumulant generating functions respectively.
One can symbolically compute the relationship between moments and cumulants 
using the relations
\[M(t)=\exp K(t)  \quad \text{ and } \quad K(t)=\log M(t).\]
For example, up to degree $5$, we obtain
\begin{equation} \label{cumulants}
  \begin{aligned}
k_0&=0,                      \\
k_1&=m_1,            \\
k_2&=m_2-m_1^2,           \\
k_3&=m_3-3m_1m_2+2m_1^3,    \\
k_4&=m_4-4m_1m_3-3m_2^2+12m_1^2m_2-6m_1^4, \\
k_5&=m_5-5m_1m_4-10m_2m_3+20m_1^2m_3+30m_1m_2^2-60m_1^3m_2+24m_1^5,       
  \end{aligned}
\end{equation}
and more generally 
$ k_d = m_d +p_d(m_1, m_2, \dots, m_{d-1})$,
where $p_d$ is some polynomial.
In particular, the first cumulant is the mean
and the second cumulant the variance of the distribution.

Notice that the cumulants are given in a triangular form
and can therefore be inverted to give the moments as functions thereof.
Precise (and in our opinion beautiful) combinatorial formulas giving the cumulants as functions of the moments can be found in Chapter 4 of \cite{zwiernik2015semialgebraic}.

Another useful property of cumulants is translation invariance:
Adding a quantity $q$ to each element in a sample only increases the first cumulant by $q$, while all higher cumulants remain the same.
One often exploits this property by assuming that 
the mean $m_1=k_1$ is zero.

\subsection{Identifiability of finite mixtures by elimination theory}
\label{sec:recoveryElim}

In this subsection we apply elimination theory \cite{Cox:2015} 
to recover the parameters of a mixture distribution.
We transform polynomials into their cumulant versions 
instead of moments because computations are significantly faster.
We present an algorithm to do this and 
explicitly perform computations
in the case of a mixture with two components 
using the computer algebra system Macaulay2 \cite{M2}. 

In this case, the $i$-th moment is given by
\[m_i = \lambda(\xi_1^i+i\alpha_1\xi_1^{i-1})+(1-\lambda)(\xi_2^i+i\alpha_2\xi_2^{i-1}).\]



Transforming into cumulants by using the equations~\labelcref{cumulants},
we obtain polynomials
\begin{equation}
\label{52}
  \begin{aligned}
    &f_1 = -k_1 + \lambda(\xi_1+\alpha_1)+(1-\lambda)(\xi_2+\alpha_2), \\
    &f_2 = -k_2 + \lambda(\xi_1^2+2\alpha_1\xi_1)+(1-\lambda)(\xi_2^2+2\alpha_2\xi_2)\\ 
    & \quad {} -(\lambda(\xi_1+\alpha_1)+(1-\lambda)(\xi_2+\alpha_2))^2,\\
    &f_3= -k_3 + \lambda(\xi_1^3+3\alpha_1\xi_1^2)+(1-\lambda)(\xi_2^3+3\alpha_2\xi_2^2)\\ 
    &\quad -3(\lambda(\xi_1+\alpha_1)+(1-\lambda)(\xi_2+\alpha_2))(\lambda(\xi_1^2+2\alpha_1\xi_1)+(1-\lambda)(\xi_2^2+2\alpha_2\xi_2)) \\
    & \quad +2( \lambda(\xi_1+\alpha_1)+(1-\lambda)(\xi_2+\alpha_2))^3,
  \end{aligned}
\end{equation}
and so on.
Instead of using elimination theory directly on the variables $\xi_i$,
we rather look at their symmetric polynomials $\xi_1 \xi_2$ and $\xi_1+\xi_2$, 
as was done in \cite{afs16:gaussians}.

Define a polynomial $f_s=s-(\xi_1+\xi_2)$ and consider the ideal generated by $\langle f_1, \dots f_5, f_s \rangle$. 
By eliminating the variables $\alpha_1,\alpha_2, \xi_1, \xi_2, \lambda$,
we obtain an ideal generated by the single polynomial
\begin{align*}
&g_s=(4k_2^3+k_3^2)s^4 \\ 
& -(32k_1k_2^3+24k_2^2k_3+8k_1k_3^2+4k_3k_4
     )s^3\\
     &+(96k_1^2k_2^3+24k_2^4+144k_1k_2^2k_3+24k_1^2k_3^2+36k_2k_3^
     2+20k_2^2k_4+24k_1k_3k_4+4k_4^2 \\
     & \quad +2k_3k_5)s^2\\
     &-(128k_1^3k_2^3+96k_1k_2^4+288k_1^2k_2^2k_3+32k_1^3k_3^2+80k_2^3k_3+144k_1k_2k_3^2+
     80k_1k_2^2k_4\\
     & \quad  +48k_1^2k_3k_4 +8k_3^3
     +40k_2k_3k_4+16k_1k_4^2+8k_2
     ^2k_5+8k_1k_3k_5+4k_4k_5)s \\
     &+(64k_1^4k_2^3+96k_1^2k_2^4+192k_1^3k
     _2^2k_3+16k_1^4k_3^2+160k_1k_2^3k_3+144k_1^2k_2k_3^2+80k_1^2k_2^
     2k_4 \\
     & \quad  +32k_1^3k_3k_4 
     +72k_2^2k_3^2+16k_1k_3^3+80k_1k_2k_3k_4+16k_
     1^2k_4^2+16k_1k_2^2k_5+8k_1^2k_3k_5 \\ 
     & \quad 
     +4k_3^2k_4 
     +16k_2k_3k_5
     +8k_1
     k_4k_5+k_5^2)
\end{align*}
that has degree 4 as a polynomial in the variable $s$ over $\K[k_1,\dots,k_5]$.



If we substitute numerical values of the first five cumulants, we can algebraically identify the parameters of the distribution.
Indeed, for every solution for $s$ of the polynomial $g_s$ above, we uniquely recover the original parameters as follows. 
First, we eliminate the variables $(\alpha_1, \alpha_2, \xi_1, \xi_2, \lambda)$ from the ideal 
$\langle f_1, \dots , f_5, f_s, f_p \rangle$,
where $f_p = p-\xi_1\xi_2$.
Inside the generators, we find the polynomial
\[(2sk_2-2k_3)p+6sk_2^2-s^2k_3-10k_2k_3+2sk_4-k_5 \]
which is linear in $p$ and thus we can substitute all values to identify $p$,
which allows us to obtain a pair of values for $\xi_1$ and $\xi_2$.

In order to determine the remaining parameters $\lambda,\alpha_1,\alpha_2$,
define 
\[\lambda_1' \coloneqq \lambda\alpha_1,\qquad \lambda_2'\coloneqq (1-\lambda)\alpha_2\]
and observe that, in these new parameters,
the moments depend only linearly on $\lambda, \lambda_1', \lambda_2'$,
as we have
\[
  m_i = \lambda \xi_1^i + \lambda_1' i \xi_1^{i-1}
  + (1-\lambda) \xi_2^i + \lambda_2' i \xi_2^{i-1}.
\]
Thus, from a computational point of view,
the main difficulty lies in finding the points $\xi_1,\xi_2$.

%

This procedure also implies that the cumulant (respectively moment) map sending $(\alpha_1, \alpha_2, \xi_1, \xi_2, \lambda)$ to $(k_1, \dots , k_5)$ is generically four-to-one. 

What is further, using the first six cumulants allows us to rationally identify the parameters. 
Indeed, let $f_6$ be the polynomial as in \ref{52} that contains the information about the sixth cumulant.
By eliminating the five variables $\alpha_1,\alpha_2, \xi_1, \xi_2, \lambda$ 
from $\langle f_1, \dots, f_6, f_s \rangle$  
we obtained an ideal that contains a polynomial of degree $1$ in $s$.
This polynomial has $412$ terms and is too long to report here.
In this case the cumulant map is one-to-one.

Using the polynomials in this section, one can substitute the values of the cumulants coming from any sample of a two component mixture and recover the parameters.
The process we describe above can potentially be adjusted for mixtures of Dirac local mixtures
with any number of components and local mixture depth.

We phrase the above process for a mixture of two first order local Diracs 
as an algorithmic strategy for parameter estimation. 
We remark here that that we write down the following primarily as an experimental process
and we do not provide a proof that it works in all cases.

\begin{algorithm}
\caption{Parameter recovery for local Dirac mixtures with $r$ components of order $l$}\label{algo1}
\begin{algorithmic}[1]
  \Require The order $l\ge 0$ of the mixture components,
  the number of mixture components $r\ge 1$
  and the moments $m_0=1,m_1,\dots,m_{(l+2)r}$.
  \Ensure The parameters $\xi_j$, $\lambda_j$ and $\alpha\derivc{jk}$
  for $1\le j\le r$, $1\le k \le l$.
\State Let $f_1, \dots , f_{(l+2)r-1}$ be polynomials whose zeros give the first $(l+2)r-1$ moments (or cumulants) as functions of the parameters,
as for example in \ref{52}.
\State Let $h_{p_i}$ be the $i$-th elementary symmetric polynomial on $\xi_j$ for $i=1, 2, \dots, r$,
so for example $h_{p_1}=\xi_1 + \dots + \xi_r$ and $h_{p_r}=\xi_1 \dots \xi_r$, and set $f_{p_i}=p_i-h_{p_i}$. 
\State Eliminate the parameter variables $\xi_j$, $\lambda_j$ and $\alpha\derivc{jk}$ from the ideal 
\[\langle f_1, \dots , f_{(l+2)r-1}, f_{p_1}\rangle\]
to obtain a polynomial $g_1$ in $\K[m_1, \dots , m_{(l+2)r-1}][p_1]$.
\State Using separately each of $f_{p_2}, f_{p_3}, \dots, f_{p_r}$, obtain polynomials $g_2, g_3, \dots, g_r$ in $\K[m_1, \dots , m_{(l+2)r-1}][p_1,p_i]$ for $i=2, 3, \dots , r$.
\State Substitute the values for the moments (or cumulants) from a sample and solve the equations to get a list of potential values for the $p_i$.
\State Use the values of the $p_i$ to obtain the $\xi_j$.
From those, the rest of the parameters can be obtained.
\State\label{algo1:weights}%
Use some method to choose the parameters that best fit the sample,
such as the additional moment $m_{(l+2)r}$.
\end{algorithmic}
\end{algorithm}
In case of a probability distribution,
a practical thing to do in Step~\labelcref{algo1:weights} would be to check, for which $((l+2)r-1)$-tuples
of $\xi_j$, $\lambda_j$ and $\alpha\derivc{jk}$, the $\lambda_j$ are real numbers between  $0$ and $1$ and discard the rest of the solutions.
Then one can check which parameter set gives an $(l+2)r$-th moment that is closest to the empirical moment $m_{(l+2)r}$.

\subsection{Prony's method}

In the following, we describe an algorithm for parameter recovery that is motivated by Prony's method.
Prony's method is a widely used tool in signal processing
that is useful for recovering sparse signals from Fourier samples
and dates back to \cite{prony1795}.
Here, we closely follow the discussion of Prony's method in \cite{mourrain17:polyexp}
because it covers the case of multiplicities.
The variant that we use is summed up in \ref{thm:pronyUnivLocal} below.

We first fix some notation.
Denote by $\K[X]_{\le a}$ the vector subspace of polynomials of degree at most $a\in\N$
and let $\K[X]\dual \coloneqq \Hom(\K[X],\K)$ be the dual $\K$-vector space of the polynomial ring $\K[X]$.
Given any sequence $(m_i)_{i\in\N}$, $m_i\in\K$,
define $\sigma\in\K[X]\dual$ to be the $\K$-linear functional
\[
  \sigma\colon \K[X] \longrightarrow \K,\quad
  X^i \longmapsto m_i.
\]
Hence, $\sigma$ is the composition of the map $\E$ of \labelcref{mapEE}
and the evaluation map $M_i \mapsto m_i$.
Further, let $M_{\sigma}$ be the $\K$-linear operator \[
  M_{\sigma}\colon \K[X] \longrightarrow \K[X]\dual,\quad
  p\longmapsto (q \mapsto \sigma(pq)).
\]
In the $\K$-vector space basis $X^i$ and the dual basis $(X^j)\dual$,
this map is described by an infinite Hankel matrix with entries $m_{i+j}$;
see \cite[Remark~2.1]{mourrain17:polyexp}.
Denote by $\mommat{a,b}$ the truncation of $M_{\sigma}$ to degrees $a,b$
\begin{equation}\label{def:momentmat}
  \mommat{a,b}\colon \K[X]_{\le b} \longrightarrow \K[X]\dual_{\le a}.
\end{equation}
The matrix $\mommat{a,b} = \lr{m_{i+j}}_{0\le i\le a,\ 0\le j\le b}$
is of size $(a+1)\times(b+1)$.
If $(m_i)_{i\in\N}$ is a sequence of moments of some distribution,
$\mommat{a,b}$ is called \emph{moment matrix}.

Assume now we are given an $r$-mixture of local Dirac mixture distributions of the form
$\sum_{j=1}^r \lambda_{j0} \dirac{\xi_j}
+ \lambda_{j1} \dirac{\xi_j}' + \cdots
+ \lambda\derivc{j,l_j} \dirac{\xi_j}\deriv{l_j}$
for some $\xi_j, \lambda\derivc{j,k_j} \in\K$, $0\le k_j\le l_j$, $1\le j\le r$.
Then its moments are of the form
\begin{equation}\label{eq:momentevalsigma}
  m_i = \sum_{j=1}^r \sum_{k_j=0}^{l_j} \lambda\derivc{j,k_j} \tfrac{\fac{i}}{\fac{(i-k_j)}} \xi_j^{i-k_j}
  = \sum_{j=1}^r (\Lambda_j(\partial)(X^i))(\xi_j) \in\K,\quad i\in\N,
\end{equation}
where
$\Lambda_j(\partial) \coloneqq \sum_{k=0}^{l_j} \lambda\derivc{j,k} \partial^k
\in \K[\partial]$
is a polynomial of degree $l_j$ in $\partial$
that is applied to the monomial $X^i$ as a differential operator.
We cite the following theorem in order to rephrase it in our language.
In particular, we specialize it to the univariate case.
We remark that the theorem can also be understood
in terms of the canonical form of a binary from.
For a detailed treatment of this viewpoint,
we refer to \cite{iarrobinokanev99} and the references therein.
\begin{Thm}[{\cite[Theorem~1.43]{iarrobinokanev99}, \cite[Theorem~4.1]{mourrain17:polyexp}}]\label{thm:pronyUnivLocal}
  Let $\K$ be an algebraically closed field of characteristic~$0$
  and let $m_0,m_1,\dots,m_{2s}\in\K$ for some $s\in\N$.
  Let $\mommat{s-1,s-1}, \mommat{s,s}$ be the corresponding Hankel matrices
  as in \labelcref{def:momentmat}.
  Assume $\rank\mommat{s-1,s-1} = \rank\mommat{s,s} = r'$.
  Then there exists a unique mixture of local mixtures of Diracs
  \[
    \mu \coloneqq \sum_{j=1}^r \Lambda_j(\partial)\dirac{\xi_j}
  \]
  for some $r\in\N$,
  $\xi_j\in\K$,
  $0\ne\Lambda_j\in\K[\partial]$,
  such that $\sum_{j=1}^r 1 + \deg(\Lambda_j) = r'$
  and its moments up to degree $2s$ coincide with $m_0,\dots,m_{2s}$.
  Further, as ideals of $\K[X]$, it holds that
  \[
    \K[X]\cdot \kernel\mommat{s-1,s} = \bigcap_{j=1}^r \idealspan{X-\xi_j}^{1+\deg{\Lambda_j}}.
  \]
\end{Thm}
Note that the condition $\sum_{j=1}^r 1 + \deg(\Lambda_j) = r'$ is due to the fact that
$\deg(\Lambda_j) + 1$ is the dimension of the vector space spanned by $\Lambda_j$ and all its derivatives,
as we specialized to the univariate case.

This leads to a two-step algorithm for recovering the parameters of $\mu$ from finitely many moments:
first, compute the points $\xi_1,\dots,\xi_r$ from $\kernel\mommat{s-1,s}$ for $s$ sufficiently large;
next, determine the weights $\Lambda_j$ from \labelcref{eq:momentevalsigma}.
If $\deg{\Lambda_j}=0$ for all $j=1,\dots,r$, this algorithm is classically known as \emph{Prony's method},
and is also referred to as \emph{Sylvester's algorithm}
in the classical algebraic geometry literature.
See for instance \cite{iarrobinokanev99,landsberg:tensors} and the references therein.


In the following, we refine this algorithm for the case of mixtures of local mixtures of Diracs
of fixed order $l \coloneqq l_1 = \cdots = l_r$ where $l_j = \deg \Lambda_j$ for $j=1,\dots,r$.
In this case, it is usually possible to recover the parameters from significantly fewer moments.

\begin{Prop}\label{thm:pronyRecovery}
  Let $\K$ be a field of characteristic~$0$ and let
  $\mu \coloneqq \sum_{j=1}^r \Lambda_j(\partial) \dirac{\xi_j}$
  be an $r$-mixture of $l$-th order local mixtures of Diracs,
  i.\,e., $\xi_j\in\K$ and $\Lambda_j\in \K[\partial]$ with $\deg(\Lambda_j) = l$, $1\le j\le r$.
  Then, the parameters $\Lambda_j, \xi_j$ of $\mu$ can be recovered from the moments
  $m_0,m_1,\dots,m_{2(l+1)r-1}$ of $\mu$ using \ref{algo2}.
\end{Prop}
\begin{proof}
  Let $\sigma$ and $M_{\sigma}$ be defined as above.
  Then, by \cite[Theorem~3.1]{mourrain17:polyexp},
  we have $\rank M_{\sigma} = (l+1)r$.
  It follows from \cite[Proposition~3.9]{mourrain17:polyexp} that
  $\rank \mommat{a,b} = (l+1)r$ for all $a,b \ge (l+1)r - 1$.
  In particular, for $s\coloneqq (l+1)r$,
  we have
  \[
    \rank\mommat{s-1,s-1} = \rank\mommat{s,s} = (l+1)r.
  \]

  The algorithm is based on the following observation.
  Let $p$ be the polynomial $p \coloneqq \prod_{j=1}^r (X-\xi_j) = X^r + \sum_{i=0}^{r-1} p_i X^i$,
  noting that knowledge of $p$ is enough to recover the points $\xi_j$.
  By the addendum of \ref{thm:pronyUnivLocal},
  we have $p^{l+1} \in (\K[x] \cdot \kernel \mommat{s-1,s}) \otimes_{\K} \algclsr{\K}$
  where $\algclsr{\K}$ is the algebraic closure of $\K$.
  Since also $p^{l+1} \in \K[x]$, it follows
  in particular that $p$ is mapped to $0$ under the composition of the maps
  \[
    \begin{tikzcd}[row sep=0.0ex]
      \K[X]_{\le r} \arrow{r}& \K[X]_{\le (l+1)r} \arrow{r} & \K[X]\dual_{\le r-1},\\%
      q \arrow[r,mapsto] & q^{l+1} \arrow[r,mapsto] & \mommat{r-1,(l+1)r} q^{l+1},
    \end{tikzcd}
  \]
  where the second map is the $\K$-linear map given by the moment matrix $\mommat{r-1,(l+1)r}$,
  which is a submatrix of $\mommat{s-1,s}$.
  The first map however is non-linear,
  defined by taking the $(l+1)$-th power of $q$ viewed as a polynomial.

  For the polynomial $p$,
  this yields the following polynomial system of $r$ equations of degree $l+1$
  in $r$ variables $p_0,\dots,p_{r-1}$ which are the monomial coefficients of $p$:
  \begin{equation}\label{eq:momentsystem}
    \mommat{r-1,(l+1)r} p^{l+1} = 0.
  \end{equation}
  By Bézout's theorem, this system of equations either has infinitely many
  or at most $(l+1)^r$ solutions.
  If the solution set is infinite, we need to add more algebraic constraints to the system
  in order to determine $p$, which is done by adding more rows to the moment matrix.

  By hypothesis, we have $\xi_1,\dots,\xi_r\in\K$. Therefore,
  termination of this algorithm and correct recovery of the points $\xi_1,\dots,\xi_r$
  follow from \ref{thm:pronyUnivLocal}.

  As for computation of the weights $\lambda\derivc{jk}$ in Step~\labelcref{algo2:vandermonde},
  note that, once the roots $\xi_j$ have been computed,
  the moments are a linear combination of the monomials $\xi_j^i$ and their derivatives
  given by \labelcref{eq:momentevalsigma},
  so to compute the weights $\lambda\derivc{jk}$,
  we solve the linear system
  \[
    \left(V_1,\dots,V_r\right)
    \begin{pmatrix}
      \lambda_{10} \\ \vdots \\ \lambda\derivc{1l} \\ \vdots \\ \lambda\derivc{rl}
    \end{pmatrix}
    =
    \begin{pmatrix}
      m_0 \\ \vdots \\ m_d
    \end{pmatrix}
  \]
  for $d\ge s$,
  where $(V_1,\dots,V_r)$ is a \emph{confluent Vandermonde matrix},
  for which each block is given by
  \[
    V_j = \left((\partial^k X^i)(\xi_j)\right)_{0\le i\le d,\atop 0\le k\le l} =
    \begin{pmatrix}
      1 & 0 & \cdots & 0 \\
      \xi_j & 1 & \cdots & 0 \\
      \vdots & \vdots & & \vdots \\
      \xi_j^d & d \xi_j^{d-1} & \cdots & \tfrac{\fac{d}}{\fac{(d-l)}} \xi_j^{d-l}
    \end{pmatrix}.
  \]
  Since the system is linear,
  uniqueness of the solution follows from the claim that
  the confluent Vandermonde matrix is of full rank $s$.
  Without loss of generality,
  we can assume the confluent Vandermonde matrix to be of size $s\times s$ by choosing a suitable submatrix.
  Then the claim follows from the fact that the Hermite interpolation problem has a unique solution
  if the points $\xi_1,\dots,\xi_r$ are distinct
  or, equivalently, from the product formula for the determinant of a square confluent Vandermonde matrix;
  see \cite[Problem~6.1.12]{hornjohnson1994:topics}.
\end{proof}

\begin{algorithm}
\caption{Parameter recovery for mixtures with $r$ components of order $l$}\label{algo2}
\begin{algorithmic}[1]
  \Require The (maximum) order $l\ge 0$ of the mixture components,
  the number of mixture components $r\ge 1$
  and the moments $m_0,\dots,m_{(l+2)r},\dots$
  \Ensure The parameters $\xi_j$ and $\lambda\derivc{jk}$
  for $1\le j\le r$, $0\le k \le l$,
  satisfying \ref{eq:momentevalsigma}.
  \State\label{algo2:system}%
  Solve the polynomial system
  $\mommat{r-1,(l+1)r} p^{l+1} = 0$ for $p$.
  \State\label{algo2:infinite}%
  If the solution set is infinite,
  increase the number of rows of the moment matrix and repeat.
  \State\label{algo2:unique}%
  If there is more than one solution, use further information,
  such as the additional moment $m_{(l+2)r}$, to restrict to a single solution $p$.
  \State Compute the roots $\xi_1,\dots,\xi_r$ of $p$.
  \State\label{algo2:vandermonde}%
  Compute the weights $\lambda\derivc{jk}$ by solving a confluent Vandermonde system.
\end{algorithmic}
\end{algorithm}

Note that the algorithm is designed to use as few moments as possible.
See also \ref{algo2:discussion} for a discussion of the number of moments used by this algorithm.

\begin{Ex}
  For $r=2$, $l=1$, let $m_0,\dots, m_5$ be the moments of a corresponding distribution
  and write $p = p_0 + p_1 X + X^2$.
  Then the system of equations \labelcref{eq:momentsystem} is given by the quadratic equations
  \begin{align}
    \label{eq:momentsystemrank2}
    \begin{pmatrix}
      m_0 & m_1 & m_2 & m_3 & m_4\\
      m_1 & m_2 & m_3 & m_4 & m_5
    \end{pmatrix}
    \begin{pmatrix}
      p_0^2 \\ 2 p_0 p_1 \\ 2 p_0 + p_1^2 \\ 2 p_1 \\ 1
    \end{pmatrix}
    = 0.
  \end{align}
  If $\xi_1, \xi_2$ are the points of the underlying distribution,
  one solution of this system is given by $p = (X-\xi_1)(X-\xi_2)$,
  that is, $p_1 = -(\xi_1 + \xi_2)$ and $p_0 = \xi_1 \xi_2$.
  Hence, computing $p_1$ by eliminating $p_0$ from System~\labelcref{eq:momentsystemrank2}, and vice versa,
  is equivalent to the process of recovering the parameters
  from elementary symmetric polynomials
  outlined in \ref{sec:recoveryElim}.
  However, with the presented new approach, we need to eliminate only a single variable instead of~$5$,
  which makes this much more viable computationally.
\end{Ex}

\begin{Rem}
  \label{algo2:discussion}
  We discuss some of the steps involved in \ref{algo2}.
  Solving the system in Step~\labelcref{algo2:system} can be done using symbolic methods,
  as outlined in the previous sections,
  or using numerical tools.
  In \ref{fourierexample} for instance,
  we use a numerical solver for this which is based on homotopy continuation methods.

  Restricting from finitely many solutions to a single one
  using the additional moment $m_{(l+2)r}$ in Step~\labelcref{algo2:unique}
  works by observing that $\mommat{r,(l+1)r} p^{l+1} = 0$.
  If a numerical solver is used, the computed solution will only be approximately zero,
  and one should assert that the selected solution is significantly closer to zero
  than all other possible choices.
  Another common approach to check uniqueness of the solution is to perform monodromy loop computations
  using a homotopy solver.

  In the worst case, this algorithm makes use of moments $m_0,\dots,m_{2(l+1)r-1}$
  as stated in \ref{thm:pronyRecovery}.
  Then solving the polynomial system in \ref{algo2} simplifies,
  since the solution is in the kernel of $\mommat{(l+1)r-1,(l+1)r}$
  which is a linear problem.
  In this case, the algorithm performs the same computation as \cite[Algorithm~3.2]{mourrain17:polyexp},
  so this guarantees termination.

  However, as \ref{algo2} solves a more specific problem,
  it can usually recover the parameters using a smaller number of moments.
  The polynomial system in Step~\labelcref{algo2:system} consists of $r$ equations of degree $l+1$ in $r$ unknowns,
  so, generically, we expect finitely many solutions in Step~\labelcref{algo2:infinite}
  already in the first iteration of the algorithm.
  This means we expect to algebraically identify
  the parameters from the moments $m_0,\dots,m_{(l+2)r-1}$
  and to rationally identify them using one additional moment,
  so usually we do not need all the moments up to $m_{2(l+1)r-1}$.
  This is also what we observe in practice, so we do not get infinitely many solutions for generic input
  if we use the moments up to $m_{(l+2)r-1}$.
  By a parameter count, we cannot expect to be able to recover the parameters from fewer moments,
  so the number of moments we use in practice is the minimal number possible.
\end{Rem}

We use the term \emph{algebraic identifiability} in the same way as in \cite{ARS},
that is, meaning that the map from the parameters to the moments is generically finite-to-one.
In this case, the \emph{identifiability degree} is the cardinality of the preimage 
of a generic point in the image of the moment map (up to permutation).
Similarly, \emph{rational identifiability} means that the moment map is generically one-to-one.

\begin{Prop}\label{algidentifiability}
  Let $d\ge (l+2)r-1$.
  Then algebraic identifiability holds for the moment map sending the parameters $\xi_j,\Lambda_j$
  to the moments $m_0,\dots,m_d$, where $\deg \Lambda_j = l$, $1\le j\le r$.
\end{Prop}
\begin{proof}
  By~\cite[Proposition~3.1]{cgg02:secanttangent}, the secant varieties of the tangent variety
  of the Veronese curve are non-defective, that is,
  for $l=1$,
  the dimension of the moment variety in $\P^d$
  for mixtures with $r$ components of order $l$
  is the expected one: $\min({3r-1}, d)$.
  In particular this means that
  the moment variety fills the ambient space sharply if $d = 3r-1$
  and does not fill the ambient space if $d > 3r-1$.
  Thus, the moment map is generically finite-to-one if $d\ge 3r-1$.
  Note that for $d < 3r-1$ the cardinality of the preimage of a generic point
  is infinite for dimension reasons.

  Similarly, for $l\ge 2$,
  the moment variety is a secant variety of the $l$-th osculating variety
  to the Veronese curve which is non-defective by \cite[Section~4]{bcgi2004:osculating},
  so the moment map is generically finite-to-one for $d\ge (l+2)r-1$.
\end{proof}


We do not currently know whether rational identifiability holds as soon as $d\ge (l+2)r$,
although we expect this to be true as discussed in \ref{algo2:discussion}.

\begin{OP}
The computations we have done in this section suggest 
that the algebraic identifiability degree for a mixture with $r$ components of order $l$ is $(l+1)^r$,
which is the expected number of solutions of \ref{eq:momentsystem}.
\end{OP}  

\begin{Rem}\label{rem:multivariatediscussion}
  It would be interesting to generalize \ref{algo2} to the multivariate setting.
  Note that in this case \cite[Algorithm~3.2]{mourrain17:polyexp} can be used
  to find the decomposition.
  However, since this does not take into account the special structure of our input,
  namely that all the mixture distributions have the same order $l$,
  this approach might use more moments than necessary.
  This is similar to the univariate case as explained in \ref{algo2:discussion}.

  Further, the algorithm in \cite[Section~6]{bernarditaufer2018}
  also computes a generalized decomposition from a given set of moments.
  This algorithm differs from our \ref{algo2} in that
  it computes parameters of \emph{any} generalized decomposition explaining the given moments,
  rather than the unique decomposition in which each term corresponds to the same order $l$.
  In the one-dimensional case, when using as few moments as possible,
  this usually leads to a non-generalized decomposition,
  which does not recover the parameters we are interested in.
  See also the related discussion in \cite[Section~7.1]{bernarditaufer2018}.
\end{Rem}

\begin{Rem}

We briefly discuss how the problem of parameter recovery
of a mixture of $1$-local mixtures simplifies,
if the mixture components
$\dirac{\xi_j} + \alpha_j \dirac{\xi_j}'$, $1\le j\le r$,
are known to differ only in the parameters $\xi_j$,
but have a constant parameter $\alpha\coloneqq \alpha_1 =\cdots =\alpha_r$.
For this, let $X$ be a distribution with moments
$\E(X^i) = \sum_{j=1}^r \lambda_j (\xi_j^i + \alpha i \xi_j^{i-1})$
with a fixed parameter $\alpha$.
Further, let $Y$ be the distribution with moments
$\E(Y^i) = \sum_{j=1}^r \lambda_j \xi_j^i$.
Then we have \[
  \E(X^i) = \E(Y^i) + \alpha i \E(Y^{i-1})
\] and conversely \[
  \E(Y^i) = \sum_{k=0}^i \frac{\fac i}{\fac k} (-\alpha)^{i-k} \E(X^k).
\]
Hence, if $\alpha$ is known,
this allows to recover the mixing distribution $Y$ from the moments of $X$.
The parameters of $Y$ can then be recovered, e.\,g., using Prony's method.

In case $\alpha$ is fixed, but unknown,
treating $\alpha$ as a variable in the moment matrix
$\mommat{r}(Y) = (\E(Y^{i+j}))_{0\le i,j\le r}$,
it can be determined as one of the roots of $\det \mommat{r}(Y)$,
which is a polynomial of degree $r(r+1)$ in $\alpha$.
\end{Rem}

\section{Applications}
\label{sec:applications}
\subsection{Moments and Fourier coefficients}

In this section, we show how the tools developed in this paper
can be applied to the problem of recovering a piecewise-polynomial function
supported on the interval $\linterval{-\pi,\pi}$
from Fourier samples; see \cite{plonka14:pronystructured}.
For this, we describe how
moments of a mixture of local mixtures arise
as the Fourier coefficients of a piecewise-polynomial function
and illustrate this numerically.
For simplicity, we focus on the case $l=1$ of piecewise-linear functions.

Let $t_j\in\interval{-\pi,\pi}$, $1\le j\le r$, be real points and
let $f\colon \linterval{-\pi,\pi} \to \C$ be the piecewise-linear function given by
\begin{equation}
  \label{def:piecewiselinear}
  f(x)\coloneqq \sum_{j=1}^{r-1} \lr{f_j + (x-t_j) f_j'} \indicator{\linterval{t_j,t_{j+1}}}(x),
\end{equation}
where $f_j, f_j'\in\C$.
In particular, splines of degree~$1$ are of this form,
but we do not require continuity here.
The Fourier coefficients of $f$ are defined to be
\[
  c_k\coloneqq \frac{1}{2\pi} \int_{-\pi}^{\pi} f(x) \e^{-\I k x} \d x,\quad k\in\Z.
\]
from which we obtain
\[
  c_k = \frac{1}{2\pi(\I k)^2}
  \sum_{j=1}^r \lr{\I k (f_j-f_{j-1} + (t_{j-1}-t_j) f_{j-1}') + (f_j' - f_{j-1}')} \e^{-\I k t_j},
\]
for $k\in\Z\setminus\{0\}$,
where $f_0,f_0',f_r,f_r' \coloneqq 0$.
Further, let
\begin{equation}\label{eq:fourierparameters}
\begin{aligned}
  \xi_j &\coloneqq \e^{-\I t_j},\\
  \lambda_j &\coloneqq \xi_j^{-s} \lr{f_j'-f_{j-1}' - \I s (f_j - f_{j-1} + (t_{j-1}-t_j) f_{j-1}')},\\
  \lambda_j' &\coloneqq \xi_j^{1-s} \I (f_j - f_{j-1} + (t_{j-1}-t_j) f_{j-1}').
\end{aligned}
\end{equation}
Assume now, we are given finitely many Fourier coefficients $c_{-s},\dots,c_{s}$ for some $s\in\N$.
Then, for $0\le k\le 2s$, $k\ne s$, we define
\begin{equation}\label{eq:fourier2moments}
  m_k \coloneqq 2\pi (\I (k-s))^2 c_{k-s}
  = \sum_{j=1}^r \lambda_j \xi_j^k + \lambda_j' k \xi_j^{k-1}.
\end{equation}
Thus, from the knowledge of Fourier coefficients $c_{-s},\dots,c_s$ of $f$,
we can compute $m_k$, $k\ne s$,
which we interpret as the moments of a mixture of $1$-local mixtures
with support points $\xi_j$ on the unit circle.
Extending the definition to $m_s$,
by construction we have
\[
  m_s \coloneqq \sum_{j=1}^r \lambda_j \xi_j^s + \lambda_j' s \xi_j^{s-1}
  = \sum_{j=1}^r f_j' - f_{j-1}' = 0.
\]
All in all, we know the moments $m_0,\dots,m_{2s}$ of this $1$-local mixture.
Recovering the parameters $\xi_j,\lambda_j,\lambda_j'$ via \ref{algo2}
generically requires the moments $m_0,\dots,m_{3r}$,
so we need to choose $2s \ge 3r$.
Subsequently retrieving the original parameters $t_j,f_j,f_j'$
from~\labelcref{eq:fourierparameters} is straightforward.

\begin{Rem}\label{rem:piecewiseperiodic}
  The piecewise-linear function $f$ is viewed as a periodic function on the interval $\linterval{-\pi,\pi}$,
  in the discussion above.
  For simplicity in presentation,
  we assumed that $f$ is constantly zero outside of $\linterval{t_1,t_r}$,
  representing a constant line segment.
  More generally, one can adapt the computation to account for an additional (non-zero) line segment there,
  without changing the number of jumping points or required samples.
  Thus, the function $f$ consists of $r$ line segments
  and has $r$ discontinuities.
\end{Rem}

\begin{Ex}\label{fourierexample}
  We apply the process described above to a piecewise-linear function
  with $r=10$ line segments on the interval $\linterval{-\pi,\pi}$.
  The parameters $t_j,f_j,f_j'$ defining the function as in \labelcref{def:piecewiselinear}
  are listed in \ref{fig:fourierparameters}.
  The random jump points $t_j$ are chosen uniformly on the interval.
  The jump heights $f_j$ and slopes $f_j'$ are chosen with respect to a Gaussian distribution.
  The function as well as the sampling data are visualized in \ref{fig:piecewiselinear_plot}.
  By Fourier transform, the Fourier coefficients carry the same information
  as the evaluations of the Fourier partial sum at equidistantly-spaced sampling points.
  The number of sampling points equals the number of Fourier coefficients needed for reconstruction,
  namely $3r+1=31$. From $2s\ge 3r$, we determine $s=15$.
  We compute the Fourier coefficients $c_{-s},\dots,c_{s}$ from the given data
  and add some noise to each of these coefficients,
  sampled from a Gaussian distribution with standard deviation $10^{-12}$.

  \begin{table}[ht]
    \centering
    \footnotesize
    \begin{tabular}{cccc}
      $j$ & $t_j$ & $f_j$ & $f_j'$ \\
      \hline
      1  &-2.814030328751694  &-0.20121264876344414  &-0.775069863870378     \\
      2  &-2.457537611167516  &-0.35221920435611676  &-0.9795392068942285    \\
      3  &-1.4536804635810938 &-0.9254256123988903   &0.26040229778962753    \\
      4  &-1.1734228328971805 &0.4482105605664995    &-0.46848914917290574   \\
      5  &-0.6568874684874002 &1.11978779941218      &-0.8808972481620518    \\
      6  &0.54049294753688    &0.3012272070859375    &0.2777255506414151     \\
      7  &1.0213620344785337  &-0.8357295816882367   &1.5239161501048377     \\
      8  &1.0930147137662223  &-0.2071744440917742   &-1.7777276640658903    \\
      9  &1.6867064885416054  &0.8681006042361324    &-2.9330595087256466    \\
      10 &2.7678373800858678  &                      &
    \end{tabular}
    \caption{The parameters of the piecewise-linear function of \ref{fourierexample}.}
    \label{fig:fourierparameters}
  \end{table}

  \begin{figure}[ht]
    \centering
    \vspace{-.3cm}
    \includegraphics{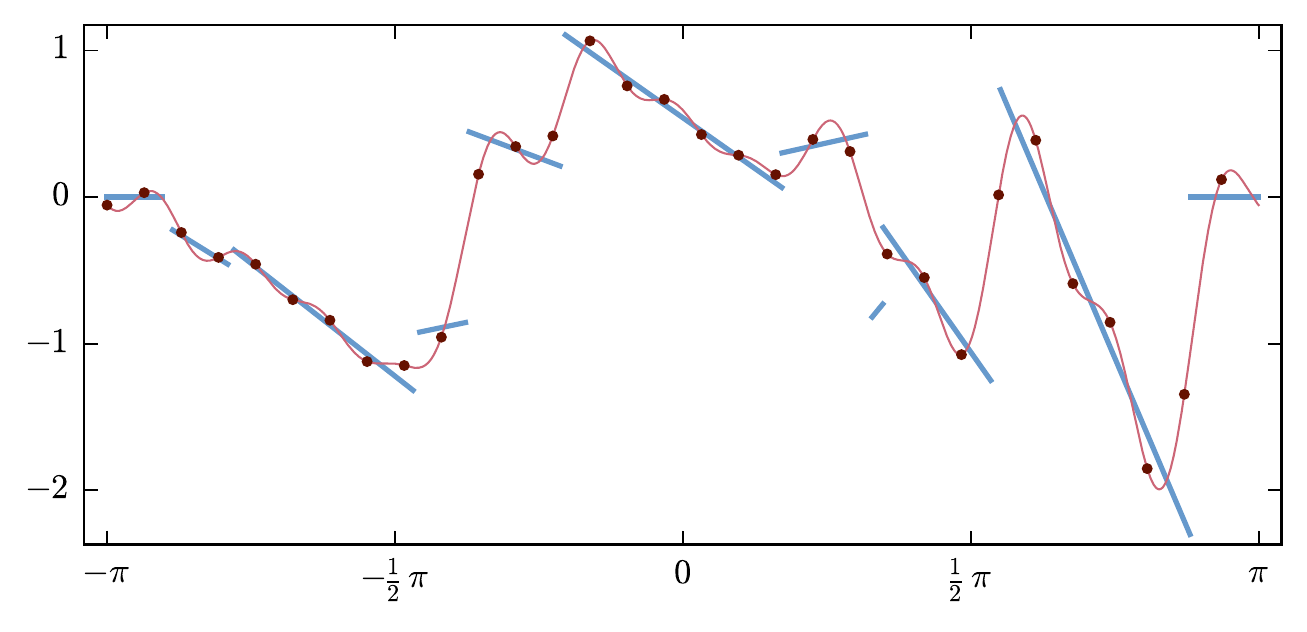}
    \vspace{-.5cm}
    \caption{The piecewise-linear function of \ref{fourierexample}
      with $r=10$ line segments;
      the Fourier partial sum approximation of order $s=15$
      and $2s+1=31$ equidistantly-spaced sampling points.}
    \label{fig:piecewiselinear_plot}
  \end{figure}

  In order to reconstruct the piecewise-linear function from the Fourier coefficients,
  we compute the moments $m_0,\dots,m_{3r}$ via \ref{eq:fourier2moments}
  and apply \ref{algo2} using numerical tools.
  From the moments $m_0,\dots,m_{3r-1}$, we get a system of $r$ quadratic equations in $r$ unknowns,
  which we solve using the Julia package \emph{HomotopyContinuation.jl}
  \cite{julia:HomotopyContinuation}, version~0.3.2,
  from which we obtain up to $2^r$ finite solutions.
  From these, we choose the one that best solves the equation system
  $\mommat{r,2r} p^2 = 0$
  induced by the additional moment $m_{3r}$.
  In this example, we obtain $1024$ solutions,
  the best of which has error $1.54\cdot 10^{-10}$ in the $\ell^2$-norm;
  the second best solution has error $3.70\cdot 10^{-4}$, which is significantly larger,
  so we accept the solution.

  Next, we compute the points $\xi_j$ using the Julia package \emph{PolynomialRoots.jl}
  \cite{julia:PolynomialRoots}, version~0.2.0,
  and solve an overdetermined confluent Vandermonde system for the weights $\lambda_j,\lambda_j'$,
  for which we use a built-in least-squares solver.
  Lastly, we use \labelcref{eq:fourierparameters} to compute the parameters $t_j,f_j,f_j'$.
  Julia code for these computations can be found in the ancillary files of the ArXiv-version of this article.
  The numerical computations were carried out using the Julia language \cite{julia2017}, version~1.0.0.

  In this example, the total error we get for the reconstructed points $t_1,\dots,t_{10}$ is
  $3.89\cdot 10^{-10}$
  in the $\ell^2$-norm, whereas for $f_j$ and $f_j'$, $1\le j\le 9$, we get
  $2.15\cdot 10^{-7}$ and
  $2.35\cdot 10^{-7}$, respectively.
  Even though, in this example, one of the line segments is quite far off of the Fourier partial sum,
  as shown in \ref{fig:piecewiselinear_plot},
  the sampling data still contains enough information to reconstruct it.
\end{Ex}

We observe that we cannot always reconstruct the randomly chosen points correctly using homotopy continuation,
but many times reconstruction is successful.
We expect that the separation distance among the points plays a major part in numerical reconstruction.
If the randomly chosen points are badly separated,
it will be difficult to distinguish them numerically by just using the moments,
as is the case if $l=0$; see \cite{moitra2015:superresolution}.

Further, we observe that, after having obtained the points,
solving the confluent Vandermonde system often induces additional errors of about three orders of magnitude,
resulting from the possibly bad condition of the confluent Vandermonde matrix.
A detailed discussion of this condition number
exceeds the scope of this paper, so we leave it for further study.

\begin{Rem}
  If $l>1$, one can adapt \ref{algo2} in a similar fashion
  to the reconstruction of functions that are piecewisely defined by polynomials of degree~$l$.
  As this requires solving a system of polynomial equations of degree~$l$,
  the involved computations are more challenging.
  Note however that,
  under additional assumptions on the smoothness of the function,
  computations can be reduced to a polynomial system of smaller degree.
  For example, if we let $l=3$ and additionally impose $\contin{1}$-continuity,
  the second derivative is piecewise-linear,
  so reconstruction can be accomplished by applying the method outlined above.
\end{Rem}

\subsection{Local mixture distributions}
\label{sec:applications2}

Local mixture models have been proposed as a means of dealing with
small variation that is unaccounted for 
when fitting data to a model, 
see \cite{anaya-izquierdo2007} 
and the references therein.
In this case the model is augmented by truncating 
a Taylor-like expansion of the probability density function.

\begin{Def}
The local mixture model of a regular exponential family  $\phi_{\xi}(x)$ is
\[\psi_{\xi}(x) := \phi_{\xi}(x) + \sum_{i=1}^l \alpha\derivc{i} \phi_{\xi}\deriv{i}(x), \]
for parameters $\alpha_1,\dots, \alpha_l$ 
such that $\psi_{\xi}(x) \geq 0$ for all $x$.
\end{Def}
The local mixture model defined this way is a convolution of a local Dirac mixture
with the member of the exponential family centered at $0$, i.\,e.,
\[
  \psi_{\xi} = \phi_0 \convolution \lr{\dirac{\xi} + \sum_{i=1}^l \alpha\derivc{i} \dirac{\xi}\deriv{i}}.
\]

\begin{Rem}\label{mgfconvolution}
Let $M_1(t), M_2(t)$ be the moment generating functions
of the two distributions in the convolution, respectively,
i.\,e.,
$M_2(t) \coloneqq \sum_{k\ge 0} m_k \frac{t^k}{\fac k}$ where
\[
  m_k = \xi^k + \sum_{i=1}^{\min\{l,k\}} (-1)^i \alpha_i \tfrac{\fac{k}}{\fac{k-i}} \xi^{k-i}.
\]
Note that the sign changes are due to the property of the derivative of the Dirac distribution that
\[
  \int \phi(x)\d\delta_{\xi}\deriv{i}(x) = (-1)^{i} \phi^{(i)}(\xi).
\]
Then the product $M_1(t) M_2(t)$ is the moment generating function corresponding to $\psi_{\xi}$,
so the moments $m_k$ of the underlying local Dirac mixture can be computed
if the moments of a random variable with density $\psi_{\xi}$
as well as the moments corresponding to $\phi_0$ are known.
Thus, we can reduce the problem of parameter inference
to the problem of parameter recovery of a local Dirac mixture.
\end{Rem}

\begin{Ex}[A mixture of two local Gaussians with known common variance]\label{ex:localgaussians}
  We numerically apply the process outlined above to a Gaussian distribution.
  For this, let $\phi_0$ be the density of a standard Gaussian distribution
  and let $\psi \coloneqq \lambda \psi_{\xi_1} + (1-\lambda) \psi_{\xi_2}$
  be a $2$-mixture of local distributions of order~$2$,
  where $\psi_{\xi_j} = \phi_{\xi_j} + \alpha_{j1} \phi_{\xi_j}' + \alpha_{j2} \phi_{\xi_j}''$.
  We choose the parameters as follows:
  $(\xi_1, \alpha_{11}, \alpha_{12}) = (-1, 0.1, 0.4)$,
  $(\xi_2, \alpha_{21}, \alpha_{22}) = (2, -0.2, 0.6)$
  and $\lambda = 0.6$.
  Note that the $\psi_{\xi_j}$ are non-negative for this choice of $\alpha_{ji}$,
  so they are indeed probability density functions; see \cite[Example~4]{marriott02:localmixtures}.
  We create a sample of size $20{,}000$ from this probability distribution
  using Mathematica \cite{mathematica11}
  and compute the empirical moments of that sample.
  Using that, we derive the empirical moments of the underlying $2$-mixture of local Dirac distributions
  as explained in \ref{mgfconvolution}
  and then apply \ref{algo2} to infer the parameters.
  We obtain the following values:
  \begin{align*}
    \xi_1 &=-0.98121, &
    \alpha_{11} &=0.14076, &
    \alpha_{12} &=0.39268, &
    \lambda &=0.59457, \\
    \xi_2 &=1.95600, &
    \alpha_{21} &=-0.20486, &
    \alpha_{22} &=0.62641.
  \end{align*}
  Note that for this process the distribution $\phi_0$ is assumed to be known.
  In particular, we need to know its standard deviation or have a way of estimating it.
  The reconstructed parameters only provide a rough approximation to the original parameters.
  Increasing the sample size can give a better approximation.
  The original distribution and the reconstructed one are shown in \ref{fig:localgaussians_plot}.
\end{Ex}

\begin{figure}[ht]
  \centering
  \vspace{-.3cm}
  \includegraphics{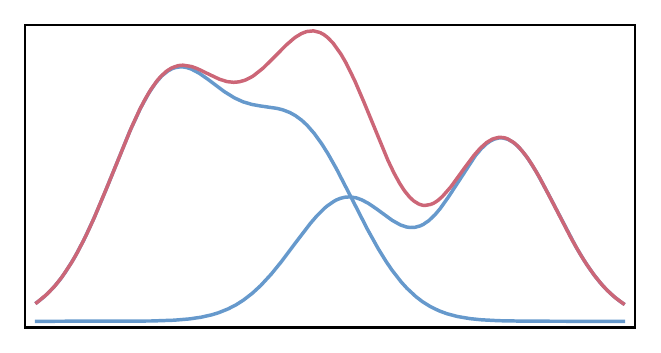}
  \includegraphics{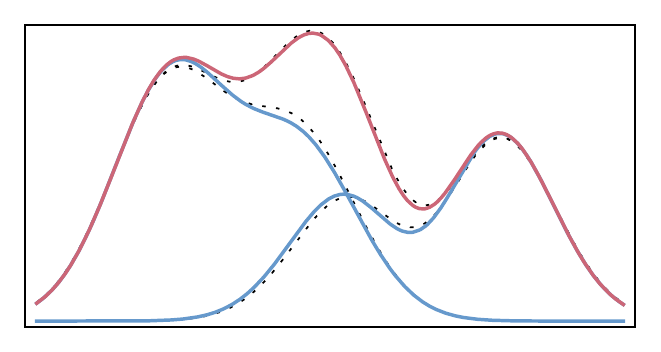}
  \vspace{-.3cm}
  \caption{The $2$-mixture of local Gaussian distributions of \ref{ex:localgaussians} on the left,
  as well as its reconstruction on the right;
  the dotted lines are the original distribution for comparison.}
  \label{fig:localgaussians_plot}
\end{figure}

\section*{Acknowledgments}
We are grateful to our advisors Stefan Kunis and Tim R\"omer for discussions and support,
as well as helpful comments on a preliminary version of this paper. 
We would like to thank the anonymous referees for their constructive comments and suggestions.
We also thank
Paul Breiding for help on using the \emph{HomotopyContinuation.jl} package.
Both authors were supported by the DFG grant GK 1916, Kombinatorische Strukturen in der Geometrie.

\setlength{\emergencystretch}{1em}
\printbibliography

\end{document}